%% file: manuel-LICS.tex
\documentclass[a4paper,11pt]{article}

\usepackage{amsmath}
\usepackage{amsthm}
\usepackage{amssymb}
\usepackage{amsfonts}
\usepackage{mathrsfs}
\usepackage{color}
\usepackage[utf8]{inputenc}
\usepackage{enumerate}
\usepackage[english]{babel}
\usepackage[bookmarks=false,colorlinks=true,linkcolor=blue,citecolor=red]{hyperref}
\usepackage[babel]{csquotes}
\usepackage{graphicx}
\usepackage{latexsym}
\usepackage{epsfig}
\usepackage{float}

\theoremstyle{plain}
\newtheorem{theorem}{Theorem}[section]
\newtheorem{lemma}[theorem]{Lemma}
\newtheorem{proposition}[theorem]{Proposition}
\newtheorem{corollary}[theorem]{Corollary}

\theoremstyle{definition}
\newtheorem{definition}[theorem]{Definition}
\newtheorem{example}[theorem]{Example}

\DeclareMathOperator{\Q}{\mathbb Q}

\DeclareMathOperator{\vcsp}{VCSP}
\DeclareMathOperator{\csp}{CSP}
\DeclareMathOperator{\dom}{dom}
\DeclareMathOperator{\feas}{Feas}

\DeclareMathOperator{\g}{\Gamma}

\DeclareMathOperator{\ar}{ar}

\def\ifinlinemath#1#2{#2}
\def\mathshift{$}
\catcode`$=13
\def${\ifinner\expandafter\mathshift\else\expandafter\myshift\fi}
\def\myshift#1${{\def\ifinlinemath##1##2{##1}\raisebox{0ex}[0ex][0ex]{\mathshift#1\mathshift}}}
\AtBeginDocument{\catcode`$=13}

\begin{document}

\title{Submodular Functions and 
	Valued Constraint \\ Satisfaction Problems over Infinite Domains}  \date{}

\author{Manuel Bodirsky \thanks{The first and second author have received funding from the European Research Council (ERC) under the European	Union's Horizon 2020 research and innovation programme (grant agreement No 681988, CSP-Infinity).}\\{\small  Institut f\"{u}r Algebra}\\{\small Technische Universit\"{a}t  Dresden} \and Marcello Mamino \footnotemark[1]\\{\small Dipartimento di Matematica}\\ {\small Universit\`{a} di Pisa} \and Caterina Viola \thanks{This author is supported by DFG Graduiertenkolleg 1763 (QuantLA).}\\{\small Institut f\"{u}r Algebra}\\ {\small Technische Universit\"{a}t  Dresden}}



\maketitle

\begin{abstract}
Valued constraint satisfaction problems (VCSPs) are a large class of combinatorial optimisation problems. It is desirable to classify the computational complexity of VCSPs depending on a fixed set of allowed cost functions in the input. Recently, the computational complexity of all VCSPs for finite sets of cost functions over finite domains has been classified in this sense. Many natural optimisation problems, however, cannot be formulated as VCSPs over a finite domain. 
We initiate the systematic investigation of infinite-domain VCSPs by studying the complexity of VCSPs for piecewise linear homogeneous cost functions. 
We show that such VCSPs can be solved in polynomial time when the cost functions are additionally submodular, and that this is indeed a maximally tractable class: adding any cost function that is not submodular leads to an NP-hard VCSP.
\end{abstract}

%
%


\section{Introduction}
In a \emph{valued constraint satisfaction problem (VCSP)} 
we are given a finite set of variables, a finite
set of cost functions that depend on these variables, and a cost $u$; the task is to find values for the variables such that the sum of the cost functions is less than $u$. 
By restricting the set of possible cost functions in the input, a great variety of computational optimisation problems can be modelled as a valued constraint satisfaction problem. By allowing the cost functions to evaluate to $+\infty$, we can even model `crisp' (i.e., hard) constraints on the variable assignments, 
and hence the class of (classical) constraint satisfaction problems (CSPs) is a subclass of the class of all VCSPs. 

If the domain is finite, the computational complexity of the VCSP has recently been classified for all sets of cost functions, assuming the Feder-Vardi conjecture for classical CSPs~\cite{KolmogorovThapperZivny15,GenVCSP15,KozikOchremiak15}. Even more recently, two solutions to the Feder-Vardi conjecture have been announced~\cite{ZhukFVConjecture,BulatovFVConjecture}. These fascinating achievements settle the complexity of the VCSP over finite domains.

Several outstanding combinatorial optimisation problems cannot be formulated as VCSPs over a finite domain, but they can be formulated
as VCSPs over the domain ${\mathbb Q}$, the set of rational numbers. 
One example is the famous linear programming problem, where the task is to optimise a linear function subject to linear inequalities. This can be modelled as a VCSP by allowing unary linear cost functions and cost functions of higher arity to express the crisp linear inequalities. Another example is 
the minimisation problem for sums of piecewise linear 
convex cost functions (see, e.g., \cite{BoydVandenberghe}). Both of these problems can be solved in polynomial time, e.g.~by the ellipsoid method (see, e.g., \cite{GroetschelLovaszSchrijver}). 

Despite the great interest in such concrete VCSPs over the rational numbers in the literature, VCSPs over infinite domains have not yet been studied systematically. 
In order to obtain general results we
need to restrict
the class of cost functions that we investigate,
because without any restriction it is already hopeless to classify the complexity of infinite-domain CSPs (any language over a finite alphabet is polynomial-time Turing equivalent to an infinite domain CSP~\cite{BodirskyGrohe}). One restriction that captures a variety 
of optimisation problems of theoretical and practical interest is the class of all \emph{piecewise linear homogeneous} cost functions over ${\mathbb Q}$,
defined below. We first illustrate by an example the type of cost functions that we want to 
capture in our framework. 

\begin{example}\label{expl:intro1}
An internet provider charges the clients depending on the amount of data $x$ downloaded and the amount of data $y$ that is uploaded. The cost function of the provider could be the partial function $f \colon {\mathbb Q}^2 \to {\mathbb Q}$ given by 
\[f(x,y) := \begin{cases} 3x & \text{ if } 0 \leq y < 2x \\
\frac{3}{2}y & \text{ if } 0 \leq 2x \leq y \\
\text{undefined} & \text{otherwise.}
\end{cases}\]
\end{example}

A partial function 
$f \colon {\mathbb Q}^n \to {\mathbb Q}$ is called
\emph{piecewise linear homogeneous (PLH)} 
if it is first-order definable over the structure ${\mathfrak L} := ({\mathbb Q};<,1,(c\cdot)_{c \in {\mathbb Q}})$; being undefined at 
$(x_1,\dots,x_n) \in {\mathbb Q}^n$ is interpreted as $f(x_1,\dots$, $x_n) = +\infty$. 
The structure $\mathfrak L$ has quantifier elimination (see Section~\ref{sect:piecewise-hom}) and hence there are finitely many regions
such that $f$ is a homogeneous linear polynomial in each region; this is the motivation for the name \emph{piecewise linear homogeneous}. 
The cost function from Example~\ref{expl:intro1} is PLH. 


The cost function in Example~\ref{expl:intro1} 
satisfies an additional important property:
it is \emph{submodular} (defined in Section~\ref{sect:submodular}). 
Submodular cost functions naturally appear in several scientific fields such as, for example, economics, game theory, machine learning, and computer vision, and play a key role in operational research and combinatorial optimisation (see, e.g.,~\cite{Fujushige}).   
Submodularity also plays an important role for 
the computational complexity of VCSPs over finite domains, and guided the research on VCSPs for some time (see, e.g., \cite{cohen2006complexity,JKT11}), 
even though this might no longer be visible in the final classification obtained in~\cite{KolmogorovThapperZivny15,GenVCSP15,KozikOchremiak15}. 

In this paper we show that VCSPs for submodular PLH cost functions can be solved in polynomial time (Theorem~\ref{thm:tract} in Section~\ref{sect:tract}). 
To solve this problem, we first describe how to solve the feasibility problem (does there exist a solution satisfying all crisp constraints) 
and then how to find the optimal solution. 
The first step follows from a new, more general polynomial-time tractability result,
namely for \emph{max-closed} PLH constraints (Section~\ref{sect:csp-tract}). 
To then solve the optimisation problem for PLH constraints, we introduce a technique to reduce the task to a problem over a finite domain that can be solved by a fully combinatorial polynomial-time algorithm for submodular set-function optimisation by Iwata and Orlin~\cite{IwataOrlin}. 
Moreover, we show that submodularity defines
a \emph{maximal tractable class}: adding any 
cost function that is not submodular leads to an NP-hard VCSP (Section~\ref{sect:maximal}). 
Section~\ref{sect:conclusion} closes with some problems and challenges.


\section{Valued Constraint Satisfaction Problems}
\label{sect:vcsps}
A \emph{valued constraint language $\Gamma$ 
(over $D$)} (or simply \emph{language}) consists of 
\begin{itemize}
\item a signature $\tau$ 
consisting of function symbols $f$, each equipped with an arity $\ar(f)$, 
\item a set $D = \dom(\Gamma)$ (the \emph{domain}), 
\item for each $f \in \tau$ a 
\emph{cost function}, i.e.,  a function $f^{\Gamma} \colon D^{\ar(f)} \to {\mathbb Q} \cup \{+\infty\}$. 
\end{itemize}
Here, $+\infty$ is an extra element with the expected properties that for all $c \in {\mathbb Q} \cup \{+\infty\}$
\begin{align*}
(+\infty) + c & = c + (+\infty) = +\infty \\
\text{  and } c & < +\infty \text{ iff } c \in {\mathbb Q}.
\end{align*} 
Let $\Gamma$ be a valued constraint language with a finite signature $\tau$. 
The \emph{valued constraint satisfaction problem  for $\Gamma$}, denoted by $\vcsp(\g)$, is the following computational problem. 

\begin{definition}\label{vcspdef}
An \emph{instance} $I$ of $\vcsp(\g)$ 
consists of 
\begin{itemize}
\item a finite set of variables $V_I$, 
\item an expression $\phi_I$ of the form
\[\sum_{i=1}^{m} f_i(x^i_1,\ldots, x^i_{\ar(f_i)})\]
where $f_1,\dots,f_m \in \tau$ and all the $x^i_j$ are variables from $V_I$, and
\item a value $u_I \in {\mathbb Q} \cup \{+\infty\}$. 
\end{itemize} 
The task is to decide whether there exists a map $\alpha \colon V_I \to \dom(\Gamma)$ whose \emph{cost}, defined as
\[\sum_{i=1}^{m} f^\Gamma_i(\alpha(x^i_1),\ldots, \alpha(x^i_{\ar(f_i)}))\]
is finite, and if so, whether there is one whose cost is smaller or equal to $u_I$.
\end{definition}
Note that since the signature $\tau$ of $\Gamma$
is finite, it is inessential for the computational complexity
of $\vcsp(\g)$ how the function symbols in $\phi_I$ are represented. 
The function described by the expression $\phi_I$ 
is also called the \emph{objective function}. 
When $u_I = +\infty$ then this problem
is called the \emph{feasibility} problem, which can
also be modelled as a (classical) constraint satisfaction problem. The choice of defining the VCSP as a decision problem and not as an optimisation problem as it is common for VCSPs over finite domain is motivated by two major issues that do not occur in the finite domain case: first, in the infinite domain setting one  needs to catch the difference between a proper minimum and an infimum value that the cost of the assignment can be  arbitrarily close  to but never reach; on the other hand, our definition allows us to model the case in which the infimum is $-\infty$, i.e., when there are assignments for the variables of arbitrarily small cost.

VCSPs have been studied intensively
when $D = \dom(\Gamma)$ is finite,
and as mentioned in the introduction, in this case a complete classification of the computational complexity of 
$\vcsp(\Gamma)$ has been obtained recently. 
However, many well-known optimisation problems can only be formulated when we allow
infinite domains $D$.

\begin{example}\label{expl:LP}
Let $\Gamma$ be the valued constraint language
with domain $D := {\mathbb Q}$ and
the signature $\tau = \{R_+,R_1,\leq,\text{id}\}$ where
\begin{itemize}
\item $R_+$ is ternary, and
\[R^\Gamma_+(x,y,z) = \begin{cases} 0 
& \text{ if } x+y=z \\
+\infty & \text{ otherwise; } \end{cases}\]
\item $R_1$ is unary and 
\[R_1^{\Gamma}(x) := \begin{cases} 0 
& \text{ if } x=1 \\
+\infty & \text{ otherwise; } \end{cases}\]
\item $\leq$ is binary and 
\[\leq^{\Gamma}(x,y)  := \begin{cases} 0 
& \text{ if } x \leq y \\
+\infty & \text{ otherwise; } \end{cases}\]
\item $\text{id}$ is unary and 
\[\text{id}^{\Gamma}(x)  := x.\]
\end{itemize}
Then instances of $\vcsp(\Gamma)$ are linear programming problems (optimise a linear objective function subject to linear inequality constraints).  
\end{example}

We give another example to illustrate the 
flexibility of the VCSP framework for formulating optimisation problems; the valued constraint language in this example contains non-convex cost functions, but, as we will see later, can be solved in polynomial time. 

\begin{example}\label{expl:intro2}
Let $\Gamma$ be the valued constraint language
with signature $\tau = \{g_1,g_2,g_3\}$ and
the cost functions 
\begin{itemize}
\item $g^\Gamma_1 \colon {\mathbb Q} \to {\mathbb Q}$ defined by $g_1(x) = -x$,
\item $g^\Gamma_2 \colon {\mathbb Q}^2 \to \mathbb Q$ defined by $g_2(x,y) := \min(x,-y)$,
and 
\item $g^\Gamma_3 \colon {\mathbb Q}^3 \to \mathbb Q$ defined by $g_3(x,y,z) := \max(x,y,z)$. 
\end{itemize}
Two examples of instances of VCSP$(\Gamma)$ are 
\begin{align} 
& g_1(x) + g_1(y) + g_1(z) + g_2(x,y) \nonumber \\
+ &  
g_3(x,y,z) + g_3(x,x,x) + g_3(x,x,x)  \label{eq:instance1} \\
\text{ and } \quad & 
g_1(x) + g_1(y) + g_1(z) \nonumber \\
+ & g_3(x,y,z) + g_3(x,x,y) + g_3(y,z,z)
\label{eq:instance2}
\end{align}
We can make the cost function described by the expression in $(\ref{eq:instance1})$ arbitrarily small by fixing $x$ to $0$ and choosing $y$ and $z$ sufficiently large. 
On the other hand, the minimum for the cost function in $(\ref{eq:instance2})$ is $0$, obtained by setting $x,y,z$ to $0$. 
Note that $g_1$ and $g_3$ are convex functions, but $g_2$ is not.
\end{example}

\section{Cost functions over the rationals}
\label{sect:cost}
The class of all valued constraint languages over arbitrary infinite domains is too large to allow for complete complexity classifications, so we have to restrict our focus to subclasses. In this section we describe natural 
and large classes of cost functions over the domain $D = {\mathbb Q}$, the rational numbers. These classes
are most naturally introduced using first-order definability, so we briefly fix the necessary logic concepts. 

\subsection{Logic Preliminaries}
We fix some standard logic terminology; see,.e.g.,~\cite{Hodges}. 
A \emph{signature} is a set $\tau$ of function and relation symbols. Each function symbol $f$ and each relation symbol $R$ is equipped with an arity $\ar(f)$, $\ar(R) \in {\mathbb N}$.  
A \emph{$\tau$-structure} $\mathfrak A$ consists of \begin{itemize}
\item a set $A = \dom(\mathfrak A)$, called the \emph{domain} of $\mathfrak A$, whose elements are called the \emph{elements} of the $\tau$-structure; 
\item for each relation symbol $R \in \tau$ 
a relation $R^{\mathfrak A} \subseteq A^{\ar(R)}$; 
\item for each function symbol $f \in \tau$ a 
function $f^{{\mathfrak A}} \colon A^{\ar(f)} \to A$. 
\end{itemize}
Function symbols of arity $0$ are allowed and
are called constant symbols. 
We give two examples of structures that play an important role in this article. 

\begin{example}\label{expl:S}
Let $\mathfrak S$ be the structure
with domain ${\mathbb Q}$ 
and the signature $\{+,1,\leq\}$
where 
\begin{itemize}
\item $+$ is a binary function symbol that 
denotes the usual addition over ${\mathbb Q}$, \item $1$ is a constant symbol that denotes $1 \in {\mathbb Q}$, and 
\item $\leq$ is a binary relation symbol that denotes the usual linear order of the rationals.
\end{itemize}  
\end{example}

\begin{example}\label{expl:L}
Let $\mathfrak L$ be the structure with
the (countably infinite) signature 
$\tau_0 := \{<,1\}\cup \{c\cdot\}_{c \in \Q}$ 
where 
\begin{itemize}
\item $<$ is a relation symbol of arity $2$
and $<^{\mathfrak L}$ is the strict linear order of $\Q$,
\item $1$ is a constant symbol and $1^{\mathfrak L} := 1 \in \Q$, and 
\item $c \cdot$ is a unary function symbol for every $c \in {\mathbb Q}$ such that $(c \cdot)^{\mathfrak L}$ is the function $x \mapsto cx$ (multiplication by $c$).
\end{itemize}
\end{example}

\subsection{Quantifier Elimination} 
Let $\tau$ be a signature. 
We adopt the usual definition of first-order logic. 
A formula is \emph{atomic} if it does not contain logical symbols (connectives or quantifiers). By convention, we have two special atomic formulas,
$\top$ and~$\bot$, to denote truth and falsity.

We say that a $\tau$-structure $\mathfrak A$ has \emph{quantifier elimination} if every first-order $\tau$-formula is equivalent to a quantifier-free $\tau$-formula over $\mathfrak A$.

\begin{theorem}[\cite{FerranteRackoff}]
The structure $\mathfrak S$ from Example~\ref{expl:S} has quantifier elimination. 
\end{theorem}

\begin{theorem}\label{thm:qe}
The structure $\mathfrak L$ from Example~\ref{expl:L} has quantifier elimination. 
\end{theorem}

Observe that every atomic $\tau_0$-formula 
has at most two variables:
\begin{itemize}
\item if it has no variables,
then it is equivalent to $\top$ or $\perp$,
\item if it has only one variable, say $x$,
then it is equivalent to $c {\cdot} x \, \sigma \, d {\cdot} 1$ or to $d {\cdot} 1 \, \sigma \, c {\cdot} x$ for $\sigma \in \{<,=\}$ and $c,d \in {\mathbb Q}$. 
Moreover, if $c = 0$ then it is equivalent to a formula without variables, and otherwise it is equivalent to $x \, \sigma \, \frac{d}{c}{\cdot} 1$
or to $ \frac{d}{c}{\cdot}1 \, \sigma \, x$ for $\sigma \in \{<,=\}$, which we abbreviate by the more common $x < \frac{d}{c}$, $x = \frac{d}{c}$, and 
$\frac{d}{c}<x$, respectively. 
\item if it has two variables, say $x$ and $y$,
then it is equivalent to $c {\cdot} x \, \sigma \, d {\cdot} y$ or $c {\cdot} x \, \sigma \, d {\cdot} y$ for $\sigma \in \{<,=\}$. Moreover, if $c = 0$ or $d = 0$ then the formula is equivalent to a formula with at most one variable, and otherwise it is equivalent to $x \, \sigma \, \frac{d}{c}{\cdot}y$ or to $\frac{d}{c}{\cdot} y \, \sigma \, x$.
\end{itemize} 

To prove Theorem \ref{thm:qe} it suffices to prove the following lemma.

\begin{lemma}\label{qel}
For every quantifier-free $\tau_0$-formula $\varphi$ there exists a quantifier-free $\tau_0$-formula $\psi$ such that $\exists x. \varphi$ is equivalent to $\psi$ over $\mathfrak L$. 
\end{lemma}

\begin{proof}
We define $\psi$ in seven steps. \begin{enumerate}
\item Rewrite $\varphi$, using De Morgan's laws, in such a way that all the negations are applied to atomic formulas. 
\item Replace \begin{itemize} 
\item $\neg (s=t)$ by $s<t\vee t<s$, and 
\item $\neg (s<t)$ by $t<s \vee s=t$,
\end{itemize} where $s$ and $t$ are $\tau_0$-terms.
\item Write $\varphi$ in disjunctive normal form in such a way that each of the clauses is a conjunction of non-negated atomic $\tau_0$-formulas (this can be done by distributivity).
\item Observe that $\exists x \bigvee_i \bigwedge_j \chi_{i,j}$, where the $\chi_{i,j}$ are atomic $\tau_0$-formulas, is equivalent to $\bigvee_i \exists x \bigwedge_j \chi_{i,j}$. Therefore, it is sufficient to prove the lemma for $\varphi=\bigwedge_j \chi_j$ where the $\chi_j$ are  atomic $\tau_0$-formulas. As explained above, we can assume without loss of generality that the $\chi_j$ are of the form $\top$, $\perp$, $x \, \sigma \, c$,  $c \, \sigma \, x$, or $x \, \sigma \, cy$, for $c \in {\mathbb Q}$ and $\sigma \in \{<,=\}$. 
If $\chi_j$ equals $\perp$, then $\varphi$
is equivalent to $\perp$ and there is nothing to be shown. If $\chi_j$ equals $\top$ then it can simply
be removed from $\varphi$. If
$\chi_j$ equals $x = c$ or $x=cy$ then replace every occurrence of $x$ by $c \cdot 1$ or by $c\cdot y$, respectively. 
Then $\varphi$ does not contain the variable $x$ anymore and thus $\exists x. \varphi$ is equivalent to $\varphi$. 
\item We are left with the case that all atomic $\tau_0$-formulas involving $x$ are (strict) inequalities, that is, $\varphi=\bigwedge_i \chi_i \wedge \bigwedge_i \chi'_i \wedge \bigwedge_i \chi''_l$, where 
\begin{itemize}
\item the $\chi_i$ are atomic formulas not containing $x$,
\item the $\chi'_i$ are atomic formulas of the form $x>u_i$,
\item the $\chi''_i$ are atomic formulas of the form $x<v_i$.
\end{itemize}
Then $\exists x. \varphi$ is equivalent to $\bigwedge_i \chi_i \wedge \bigwedge_{i,j}(u_i<v_j)$.
\end{enumerate}
Each step of this procedure preserves the satisfying assignments for $\varphi$ and the resulting formula is in the required form; 
this is obvious for all but the last step, 
and for the last step follows from the correctness of Fourier-Motzkin elimination for systems of linear inequalities. Therefore the procedure is correct.
\end{proof}

\begin{proof}[Proof (of Theorem \ref{thm:qe})]
Let $\varphi$ be a $\tau_0$-formula. We prove that it is equivalent to a quantifier-free $\tau_0$-formula by induction on the number $n$ of quantifiers of $\varphi$. For $n=1$ we have two cases:
\begin{itemize}
\item If $\varphi$ is of the form $\exists x. \varphi'$ (with $\varphi'$ quantifier-free) then, by Lemma \ref{qel}, it is equivalent to a quantifier-free $\tau_0$-formula $\psi$.
\item If $\varphi$ is of the form $\forall x. \varphi'$ (with $\varphi'$ quantifier-free), then it is equivalent to $\neg \exists x. \neg \varphi'$. By Lemma \ref{qel}, $\exists x. \neg \varphi'$ is equivalent to a quantifier-free $\tau_0$-formula $\psi$. Therefore, $\varphi$ is equivalent to the quantifier-free $\tau_0$-formula $\neg \psi$.
\end{itemize}
Now suppose that $\varphi$ is of the form
$Q_1 x_1 Q_2x_2 \cdots Q_nx_n.\varphi'$ 
for $n \geq 2$ and $Q_1,\dots,Q_n \in \{\forall,\exists\}$, and suppose that the statement 
is true for $\tau_0$-formulas with at most $n-1$ quantifiers. In particular, 
$Q_2x_2 \cdots Q_nx_n.\varphi'$ is equivalent to a quantifier -free $\tau_0$-formula $\psi$. Therefore, $\varphi$ is equivalent to $Q_1x_1.\psi$, that is, a $\tau_0$-formula with one quantifier that is equivalent to a quantifier-free $\tau_0$-formula, again by inductive hypothesis.
\end{proof}

\subsection{Piecewise Linear Homogeneous Functions}
\label{sect:piecewise-hom}
A \emph{partial function} of arity $n \in {\mathbb N}$ over a set $A$ is a function
\[f \colon \dom(f) \to A \text{ for some } \dom(f) \subseteq A^{n} \; .\]
Let $\mathfrak A$ be a $\tau$-structure. 
A partial function over $A$ is called \emph{first-order definable over $\mathfrak A$} if 
there exists a first-order $\tau$-formula $\phi(x_0,x_1,\dots,x_n)$ such that for all $a_1,\dots,a_n \in A$ 
\begin{itemize}
\item if $(a_1,\dots,a_n) \in \dom(f)$ then $\mathfrak A \models \phi(a_0,a_1,\dots,a_n)$
if and only if\\ $a_0 = f(a_1,\dots,a_n)$, and
\item if $f(a_1,\dots,a_n) \notin \dom(f)$ then there is no $a_0 \in A$ such that\\ $\mathfrak A \models \phi(a_0,a_1,\dots,a_n)$. 
\end{itemize}

In the following, 
we consider \emph{cost functions
over ${\mathbb Q}$}, which will be  
functions from ${\mathbb Q}^n \to \Q \cup \{+\infty\}$.
It is sometimes convenient to view a cost function 
as a partial function over ${\mathbb Q}$. 
 If $t \in \Q^{\ar(f)} \setminus \dom(f)$ we interpret this as $f(t) = +\infty$.

\begin{definition}
A cost function $f \colon {\mathbb Q}^n \rightarrow \Q \cup \{+\infty\}$ (viewed as a partial function)
is called 
\begin{itemize}
\item \emph{piecewise linear (PL)} if it is 
first-order definable over $\mathfrak S$, piecewise linear functions are sometimes called \emph{semilinear} functions;
\item \emph{piecewise linear homogeneous (PLH)} if it is 
first-order definable over $\mathfrak L$ (viewed as a partial function). 
\end{itemize}
A valued constraint language $\Gamma$ is called \emph{piecewise linear} (\emph{piecewise linear homogeneous}) if every cost function in $\Gamma$ is PL (or PLH, respectively). 
\end{definition}

Every piecewise linear homogeneous
cost function is also piecewise linear, since
all functions of the structure $\mathfrak L$ are clearly first-order definable in $\mathfrak S$. 
The cost functions in the valued constraint language from Example~\ref{expl:intro2} are PLH.
The cost functions in the valued constraint language from Example~\ref{expl:LP} are PL, but not PLH.

We would like to point out that already the class of 
PLH cost functions is very large. In particular, one can view it as a generalisation of the class of all sets of cost functions over a finite domain $D$. Indeed,
every VCSP for a valued constraint language over a finite domain is also a VCSP for a language that is PLH. To see this, suppose that $f \colon D^d \to \Q \cup \{+\infty\}$ is such a cost function,
identifying $D$ with a subset of $\Q$ in an arbitrary way. Then the function $f' \colon \Q^d \to \Q \cup \{+\infty\}$
defined by $f'(x_1,\dots,x_n) := f(x_1,\dots,x_n)$
if $x_1,\dots,x_n \in D$, and
$f'(x_1,\dots,x_n) = +\infty$ otherwise, 
is PLH.

\subsection{Submodularity}
\label{sect:submodular}

Let $D$ be a set. When $x^1,\dots,x^k \in D^n$ and $g \colon D^k \to D$ is a function,
then $g(x^1,\dots,x^k)$ denotes
the $n$-tuple obtained by applying $g$ \emph{componentwise}, i.e., 
\[g(x^1,\dots,x^k) := (g(x^1_1,\dots,x^k_1),\dots,g(x^1_n,\dots,x^k_n)) .\]

\begin{definition}
Let $D$ be a totally ordered set and let $G$ be a totally ordered Abelian group. 
A partial function $f \colon D^n \to G$ is called \emph{submodular} if for all 
$x$, $y \in D^n$ 
\[f(\max(x,y))+f(\min(x,y))\leq f(x)+f(y).\]
\end{definition}
Note that in particular if $ x, y \in \dom(f)$, then $\min( x, y) \in \dom(f)$ and $\max( x, y) \in \dom(f)$. Hence, the cost function $R_+$ from Example~\ref{expl:LP} is not submodular. 
On the other hand, all cost functions in Example~\ref{expl:intro2} are submodular.

\input CSP-tract.tex
\input VCSP.tex

\section{Maximal Tractability}
\label{sect:maximal} 
A \emph{sublanguage} of a valued constraint language $\Gamma$ is a valued constraint  language that can be obtained from $\g$ by dropping some of the cost functions.

\begin{definition}
Let $\mathcal V$ be a class of valued constraint languages over a fixed domain $D$ and let $\g$ be a language of $\mathcal V$. We say that $\g$ is \emph{maximally tractable within} $\mathcal V$ if \begin{itemize}
\item $\vcsp(\g')$ is polynomial time solvable for every finite sublanguage $\g'$ of $\g$; and
\item for every valued constraint language $\Delta$ in $\mathcal V$ properly containing $\g$, there exists a finite sublanguage $\Delta'$ of $\Delta$ such that $\vcsp(\Delta')$ is NP-hard.
\end{itemize}
\end{definition}

We will make use of the following result.

\begin{theorem}[Cohen-Cooper-Jeavons-Krokhin, \cite{cohen2006complexity}, Theorem 6.7]\label{ccjk2}
	Let $D$ be a finite totally ordered set. 
	Then the valued constraint language consisting of all submodular cost functions over $D$ is maximally tractable within the class of valued constraint languages over $D$. 
\end{theorem}

We show that the class of submodular piecewise linear homogeneous languages is maximally tractable within the class of PLH valued constraint languages.

\begin{definition}\label{defchi}
	Given a finite set $D \subset \Q$, we define the partial function $\chi_D \colon \Q \rightarrow \Q$ by \[\chi_D(x)=\begin{cases}
	0 & x\in D\\
	+\infty & x \in \Q \setminus D.
	\end{cases}\]
\end{definition}

For every finite set $D\subset \Q$, the cost function $\chi_D$ is submodular and PLH. 

\begin{definition}
	Given a finite domain $D \subset \Q$ and a partial function $f\colon D^n \rightarrow \mathbb Q$ we define the \emph{canonical extension} of $f$ as $\hat{f}\colon \Q^n \rightarrow \Q$, by \[\hat f(x)=\begin{cases}
	f( x) & x \in D^n\\
	+\infty & \text{otherwise}.
	\end{cases}\]
\end{definition}

Note that the canonical extension of a submodular function over a finite domain is submodular and PLH. 

\begin{theorem}\label{th:mt}
The valued constraint language consisting of all submodular PLH cost functions is maximally tractable within the class of PLH valued constraint languages. 
\end{theorem}

\begin{proof}
	Polynomial-time tractability of the VCSP for finite sets of submodular PLH cost functions has been shown in Theorem~\ref{thm:tract}. 
	
	Now suppose that $f$ is a cost function over ${\mathbb Q}$ that is not submodular, i.e., 
	there exists a couple of points, $ a:=(a_1,\ldots,a_n)$, $ b:=(b_1,\ldots,b_n) \in \Q^n$ such that \[f( a)+f( b)<f(\min( a,  b))+f(\max( a,  b)).\]
	Let $\Gamma_D$ be the language of all submodular functions on \[D:=\{a_1,\ldots,a_n,b_1,\ldots,b_n\}\subset \Q.\] 
	Notice that $f|_D$ is not submodular, for our choice of $D$.
	Therefore, by Theorem \ref{ccjk2}, there exists a finite language $\Gamma'_D \subset \Gamma_{D}$ such that $\vcsp(\Gamma'_D \cup \{f|_D\})$ is NP-hard.
	
	We define the finite submodular PLH language $\Gamma'$ by replacing every cost function $g$ in $\Gamma'_{D}$ by its canonical extension $\hat g$. 
	Then $\Gamma' \cup \{f,\chi_D\}$, where $\chi_D$ is defined as in Definition \ref{defchi}, has an NP-hard VCSP. 
	Indeed, for every instance  $I$ of $\vcsp(\Gamma'_{D}\cup\{f|_D\})$, we define an instance $J$ of $\vcsp(\Gamma' \cup \{f,\chi_D\})$ in the following way: 
	\begin{itemize}
		\item replace every function symbol $g$ in $\phi_I$ by the symbol for its canonical extension,
		\item replace the function symbol for ${f}|_{D}$ 
		in $\phi_I$ by $f$, and
		\item add to $J_\phi$ the summand $\chi_D(v)$ for every variable $v \in V_I$.
	\end{itemize}
	Because of the terms involving $\chi_D$, 
	the infimum of $\phi_J$ is smaller than $+\infty$ if, and only if, it is attained in a point having coordinates in $D$. Therefore, the infimum of $\phi_J$ coincides with the infimum of $\phi_I$. 
	Since $J$ is computable in polynomial-time from $I$, the NP-hardness of $\vcsp(\Gamma' \cup \{f,\chi_D\})$ follows from the NP-hardness of 
	$\vcsp(\Gamma'\cup\{f|_D\})$. 
\end{proof}

\section{Conclusion and Outlook}
\label{sect:conclusion}
We have presented a polynomial-time algorithm for submodular PLH cost functions over the rationals. 
In fact, our algorithm not only decides the feasibility problem and whether there exists a solution
of cost at most $u_I$, but it can also be adapted to efficiently compute 
the infimum of the cost of all solutions (which might be $-\infty$),
and decides whether the infimum is attained. 
The modification is straightforward observing that the sample computed does not depend on the threshold $u_I$. 

We also showed that submodular PLH cost functions are \emph{maximally tractable} within the class of PLH cost functions. 
Such maximal tractability results are of particular importance for the more ambitious goal to classify
the complexity of the VCSP for all classes of
PLH cost functions: to prove a complexity dichotomy it suffices to identify \emph{all} maximally tractable classes. 

Another challenge is to extend our tractability
result to the class of all submodular \emph{piecewise linear} VCSPs. We believe that submodular 
piecewise linear VCSPs are in P, too.   
But note that already the 
structure $(\Q;0,S,D)$ 
where $S := \{(x,y) \mid y = x+1\}$
and $D := \{(x,y) \mid y = 2x\}$ 
(which has both min and max as a polymorphism)
does \emph{not} admit
an efficient sampling algorithm (it is easy to see that for every $d \in {\mathbb N}$ every $d$-sample must have exponentially many vertices in $d$), so a different approach than the approach in this paper is needed.  

\bibliographystyle{abbrv} 
\bibliography{local.bib}


\end{document}

%% file: CSP-tract.tex
\section{Tractability of Max-Closed PLH Constraints}
\label{sect:csp-tract}

The question whether an instance of $\vcsp(\Gamma)$ is \emph{feasible},
namely has a solution, can be viewed as a (classical) constraint satisfaction
problem. Formally, the constraint 
satisfaction problem for a structure $\mathfrak A$ with a finite
relational signature $\tau$ is the following computational problem,
denoted by $\csp(\mathfrak A)$: 
\begin{itemize}
\item the input is a finite conjunction $\psi$ of atomic $\tau$-formulas,
and 
\item the question is whether 
$\psi$ is satisfiable in $\mathfrak A$. 
\end{itemize}
We can associate to $\Gamma$ the following
relational structure $\feas(\Gamma)$: for every cost function $f$ of arity
$n$ from $\Gamma$ the 
signature of $\feas(\Gamma)$ contains a 
relation symbol $R_f$ of arity $n$ such that 
$R_f^{\feas(\Gamma)} = \dom(f)$. 

Every polynomial-time algorithm for
$\vcsp(\g)$, in particular, has to solve $\csp(\feas(\g))$.
In fact, an instance $\phi$ of $\vcsp(\g)$ can be translated into an
instance $\psi$ of $\csp(\feas(\g))$ by replacing 
subexpressions of the form 
$f(x_1,\dots,x_n)$ in $\phi$ by 
$R_f(x_1,\dots,x_n)$
and by replacing $+$ by $\wedge$. 
It is easy to see that $\phi$ is a feasible instance of $\vcsp(\Gamma)$ if
and only if $\psi$ is satisfiable in $\feas(\g)$.

\begin{definition}
Let $\mathfrak A$ be a structure with relational signature~$\tau$
and domain~$A$.
Then a function $g \colon A^k \to A$ is called
a \emph{polymorphism} of $\mathfrak A$
if for all $R \in \tau$ we have that $R^{\mathfrak A}$ is {\it preserved} by~$g$, namely
$g(x^1,\ldots,  x^k) \in R^{\mathfrak A}$
for all $ x^1,\ldots,  x^k \in R ^{\mathfrak A}$ (where
$g$ is applied component-wise).
\end{definition}

If we consider $\g$ to be a finite submodular PLH valued
constraint language, then~$\feas(\g)$ is
a relational structure all of whose
relations are
\begin{itemize}
\item first-order definable over ${\mathfrak L}$, and
\item preserved by the polymorphisms $\max$ and $\min$.
\end{itemize}
We observed that for the polynomial time tractability of~$\vcsp(\g)$ we
need, in particular, that $\csp(\feas(\g))$ be tractable.
In this section we prove a more general result:
\begin{theorem}\label{thm-csp}
Let $\mathfrak A$ be a structure having domain~$\Q$ and finite
relational signature~$\tau$. Assume that for all $R\in\tau$, the
interpretation $R^{\mathfrak A}$ is PLH and preserved by~$\max$.
Then $\csp(\mathfrak A)$ is polynomial-time solvable.
\end{theorem}
This result is incomparable to known results
about max-closed semilinear relations~\cite{BodirskyMaminoTropical}.
In particular, there, the weaker bound~$\text{NP}\cap\text{co-NP}$ has been
shown for a larger class, and polynomial tractability only for a smaller
class (which does not
contain many max-closed PLH relations, for instance $x \geq \max(y,z)$).

We use a technique introduced in~\cite{BodMacpheTha}
which relies on the following concept.

\begin{definition}
Let $\mathfrak A$ be a structure with a finite relational signature
$\tau$. A \emph{sampling algorithm} for $\mathfrak A$ takes as input a
positive integer $d$ and computes a finite $\tau$-structure $\mathfrak B$
such that every finite conjunction of atomic $\tau$-formulas having at
most $d$ distinct free variables is
satisfiable in $\mathfrak A$ if, and only if, it is satisfiable in
$\mathfrak B$.
A sampling algorithm is called \emph{efficient} if its running time is
bounded by a polynomial in $d$.
\end{definition}

The definition above is a slight re-formulation of Definition~2.2
in~\cite{BodMacpheTha},
and it is easily seen to give the same results using the same proofs.
We decided to bound the number of variables
instead of the size of the conjunction of atomic~$\tau$-formulas because
this is more natural in our context. These two quantities are polynomially
related by the assumption that the signature~$\tau$ be finite.

\begin{definition}
A $k$-ary function $g \colon D^k \to D$ is called \emph{totally
symmetric} if
$g(x_1,\ldots,x_k)=g(y_1,\ldots,y_k)$
for all $x_1,\ldots,x_k,y_1,\ldots,y_k \in D$
such that $\{x_1\ldots,x_k\}=\{y_1,\ldots,y_k\}$.
\end{definition}

\begin{theorem}[Bodirsky-Macpherson-Thapper, \cite{BodMacpheTha}, Theorem 2.5]
\label{macphe}
Let $\mathfrak A$ be a structure over a finite relational
signature with totally symmetric polymorphisms of all arities. If there
exists an efficient sampling algorithm for $\mathfrak A$ then
$\csp({\mathfrak A})$ is in P.
\end{theorem}

\def\mfa{{\mathfrak A}}
In this section, we study $\csp(\mfa)$, where $\mfa$ is a $\tau$-structure
satisfying the
hypothesis of Theorem~\ref{thm-csp}. We give a formal definition of
the {\it numerical data} in~$\mfa$, we will need it later on. By quantifier
elimination (Theorem~\ref{thm:qe}), we can write each of the finitely many
relations $R^\mfa$ for~$R\in\tau$ as a quantifier-free $\tau_0$-formula~$\phi_R$. As
in the proof of Theorem~\ref{thm:qe}, we can assume that all formulas~$\phi_R$
are positive (namely contain no negations).
From now on, we will fix one such representation. Let
\def\atm#1{{\text{At}(#1)}}$\atm{\phi_R}$ denote the set of atomic subformulas of~$\phi_R$. Each atomic
$\tau_0$-formula is of the form $t_1 {<\atop =} t_2$, where $t_1$
and~$t_2$ are terms. We call the atomic formula non-trivial if
it is not equivalent to~$\bot$ or~$\top$,
from now on we make the following assumptions on the atomic
formulas (cf.\ again the proof of Theorem~\ref{thm:qe})
\begin{itemize}
\item that atomic formulas except~$\top,\bot$ are non-trivial
\item that the functions $k\cdot$ are never composed,
because $k\cdot h\cdot x$ can be replaced by $(kh)\cdot x$
\item that, in any atomic formula $k\cdot x_i {<\atop =} h\cdot x_j$, the
constants~$k$ and~$h$ are not both negative.
\end{itemize}

Given a set of non-trivial atomic formulas~$\Phi$, we define
\[ H(\Phi) = \left\{ \frac{c_1}{c_2} \;\middle|\; t_1=c_1\cdot x_i, \; t_2
= c_2\cdot x_j,\;  \text{for some $t_1 {<\atop =} t_2$ in $\Phi$} \right\}
\]
\begin{align*}
K(\Phi) = &\left\{ \frac{c_2}{c_1} \;\middle|\; t_1=c_1\cdot x_i, \; t_2 =
c_2\cdot 1,\;  \text{for some $t_1 {<\atop =} t_2$ in $\Phi$} \right\} \\
\cup
&\left\{ \frac{c_1}{c_2} \;\middle|\; t_1=c_1\cdot 1, \; t_2 = c_2\cdot
x_j,\;  \text{for some $t_1 {<\atop =} t_2$ in $\Phi$} \right\}
\end{align*}

We describe now the numeric domain~$\Q^\star$ in which our algorithm operates.

\def\eps{{\boldsymbol\epsilon}}
\begin{definition}
We call~$\Q^\star$ the ordered $\Q$-vector space
\[ \Q^\star = \{ x + y{\boldsymbol\epsilon} \;\mid\; x,y\in\Q\} \]
where $\boldsymbol\epsilon$ is merely a formal device, namely
$x+y{\boldsymbol\epsilon}$ represents the pair~$(x,y)$.
We define addition and multiplication by a scalar component-wise
\begin{align*}
(x_1 + y_1\eps)\,+\,(x_2+y_2\eps) &= (x_1+x_2) + (y_1+y_2)\eps \\
c\cdot(x + y\eps) &= (cx) + (cy) \eps \textbf{.}
\end{align*}
The order is induced by~$\Q$ extended with $0 < \boldsymbol\epsilon \ll 1$,
namely the lexicographical order of the components $x$ and~$y$
\[
(x_1 + y_1\eps)\,<\,(x_2+y_2\eps)\quad\text{iff}\quad
\begin{cases}
x_1 < x_2 \quad \text{or} \\
x_1 = x_2 \;\wedge\; y_1 < y_2\text{.}
\end{cases}
\]
\end{definition}

Any $\tau_0$-formula has an obvious interpretation in any
ordered $\Q$-vector space~$Q$ extending~$\Q$, and, in particular,
in~$\Q^\star$.

\begin{proposition}\label{equiv}
Let $\phi(x_1\dotsc x_d)$ and $\psi(x_1\dotsc x_d)$ be $\tau_0$-formulas.
Then $\phi$ and $\psi$ are equivalent in~$\Q$ if, and only if, they are
equivalent in any ordered $\Q$-vector space~$Q$ extending~$\Q$ (for instance $Q=\Q^\star$).
\end{proposition}
\begin{proof}
It follows from
\cite[Chapter~1, Remark~7.9]{VanDenDries} that the first-order theory of ordered $\Q$-vector
spaces in the signature~$\tau_0\cup\{+,-\}$ is complete. As a consequence
the formula $\forall x_1\dotsc x_d \phi(x_1\dotsc x_d) \leftrightarrow
\psi(x_1\dotsc x_d)$ holds in $\Q$ if and only if it does in~$Q$.
\end{proof}

The proposition gives us a natural extension $\mfa^\star$ of~$\mfa$ to the
domain~$\Q^\star$. Namely the $\tau$-structure obtained
interpreting each relation symbol $R\in\tau$ by the
relation~$R^{\mfa^\star}$ defined on~$\Q^\star$ by the same (quantifier-free)
$\tau_0$-formula~$\phi_R$ that defines $R^\mfa$ over~$\Q$ (by the
proposition, the choice of equivalent $\tau_0$-formulas is immaterial).
Similarly, we will see that, as long as satisfiability is concerned, there
is no difference between $\mfa$ and~$\mfa^\star$.

\begin{corollary}\label{transfer}
Let $\phi$ be an instance of $\csp(\mfa)$, and let~$\phi^\star$ be the
corresponding instance of~$\csp(\mfa^\star)$. Then $\phi$ is satisfiable if and only if~$\phi^\star$ is.
\end{corollary}
\begin{proof}
From Proposition~\ref{equiv} observing that $\phi$ (resp.~$\phi^\star$) is unsatisfiable if and only if it is equivalent to~$\bot$.
\end{proof}

As a consequence, we can work in the extended
structure~$\mfa^\star$. Our goal is to prove the following theorem.

\begin{theorem}\label{thm:sampling}
There is an efficient sampling algorithm for $\mfa^\star$.
\end{theorem}

Assuming, for a moment, Theorem~\ref{thm:sampling}, it is easy to prove
Theorem~\ref{thm-csp}.

\begin{proof}[Proof of Theorem~\ref{thm-csp}]
By Proposition 4.7, for all $k \geq 1$ the function
\[(x_1,\ldots,x_k)\mapsto
\max(x_1,\ldots,x_k)\]
is a $k$-ary totally symmetric polymorphism of
$\csp(\mfa^\star)$. Therefore, $\csp(\mfa^\star)$ is in P by
Theorem~\ref{thm:sampling} and Theorem~\ref{macphe}. Finally, by
Corollary~\ref{transfer}, $\csp(\mfa^\star)$ and~$\csp(\mfa)$ are equivalent.
\end{proof}

Let $\phi$ be an atomic $\tau_0$-formula. We write $\bar \phi$ for the formula $t_1 \leq t_2$ if $\phi$ is of the form $t_1 < t_2$, and for the formula $t_1 = t_2$ if $\phi$ is of the form $t_1 = t_2$.

\begin{lemma}\label{closure}
Let $\Phi$ be a finite set of atomic $\tau_0$-formulas having free variables
in~$\{x_1\dotsc x_d\}$. Assume that $\bar \Phi:=\bigcup_{\phi\in \Phi} \bar\phi$
has a simultaneous solution~$(x_1\dotsc x_d)\in\Q^{>0}$ in positive
numbers. Then $\bar \Phi$ has a solution taking values in the
set $C_{\Phi,d}\subset \Q$ defined as follows
\[
C_{\Phi,d} = 
\left\{ |k|\prod_{i=1}^s |h_i| ^{e_i} \,\middle|\, k\in K(\Phi) ,\;
e_1\dotsc e_s \in {\mathbb Z},\; \sum_{r=1}^s \lvert e_r \rvert < d
\right\}
\]
where $h_1\dotsc h_s$ is an enumeration of the (finitely many) elements
of~$H(\Phi)$.
\end{lemma}
\begin{proof}
Let $\gamma\le\beta$ be maximal such that there are $\Psi_1,\Psi_2,\Psi_3$ with
\begin{align*}
\bar \Phi \mkern 7mu &= \{s_1 = s'_1, \dotsc, s_\alpha = s'_\alpha\} \cup
\{ t_1 \le t'_1, \dotsc, t_\beta \le t'_\beta \}\\
\Psi_1 &= \{s_1 = s'_1, \dotsc, s_\alpha = s'_\alpha\}\\
\Psi_2 &= \{ t_1 = t'_1, \dotsc, t_\gamma=t'_\gamma \}\\
\Psi_3 &= \{t_{\gamma+1} \le t'_{\gamma+1},\dotsc, t_\beta \le t'_\beta \}\text{, }
\end{align*}
where $s_i, s'_i, t_j, t'_j$ are $\tau_0$-terms for all $i$, $j$, and $\Psi_1 \cup \Psi_2 \cup \Psi_3$ is satisfiable in positive numbers.
Clearly the space of positive solutions of~$\Psi_1\cup \Psi_2$ must be contained in
that of~$\Psi_3$. In fact, by construction, they intersect: consider any straight
line segment connecting a solution of $\Psi_1\cup \Psi_2\cup \Psi_3$ and a solution of
$\Psi_1\cup \Psi_2$ not satisfying~$\Psi_3$, on this segment there must be a solution of
$\Psi_1\cup \Psi_2\cup \Psi_3$ lying on the boundary of one of the inequalities of~$\Psi_3$,
contradicting the maximality of~$\gamma$.
By the last observation it suffices to prove that there is a solution
of~$\Psi_1\cup \Psi_2$ taking values in~$C_{\Phi,d}$. Put an edge between two variables $x_i$ and~$x_j$ when
they appear in the same formula of~$\Psi_1\cup \Psi_2$. For each connected component
of the graph thus defined, either
it contains at least one variable~$x_i$ such that there is
a constraint of the form~$h\cdot x_i = k\cdot 1$, 
or all constraints are of the form~$h\cdot x_i = h'\cdot x_j$.
In the first case assign $x_i=\frac{k}{h}$, in the second assign one of the
variables~$x_i$ arbitrarily to~$1$, then, in any case,
since the diameter of the connected component
is~$<d$, all variables in it are forced to take values in~$C_{\Phi,d}$ by
simple propagation of~$x_i$.
\end{proof}

\begin{lemma}\label{pippo}
Let $\Phi$ be a finite set of atomic $\tau_0$-formulas having free variables
in~$\{x_1\dotsc x_d\}$. Assume that the formulas in~$\Phi$ are simultaneously
satisfiable. Then they are simultaneously satisfiable in
$D_{\Phi,d}:=-C^{\star}_{\Phi,d} \cup \{0\} \cup C^{\star}_{\Phi,d}$ where
\[
C^{\star}_{\Phi,d} = \{x + nx\boldsymbol\epsilon \;|\; x \in C_{\Phi,d},\; n\in\mathbb Z,\; -d\le
n\le d\}
\]
$C_{\Phi,d}$ is defined as in Lemma~\ref{closure}, and $-C^{\star}_{\Phi,d}$
denotes the set~$\{-x\;|\;x\in C^{\star}_{\Phi,d}\}$.
\end{lemma}
\begin{proof}
First we fix a solution $x_i=a_i$ for $i=1\dotsc d$ of~$\Phi$. In general,
some of the values~$a_i$ will be positive, some~$0$, and some negative: we
look for a new solution $z_1\dotsc z_d\in D_{\Phi,d}$ such that $z_i$ is
positive, respectively~$0$ or negative, if and only if $a_i$ is.

To this aim we rewrite the formulas in~$\Phi$ replacing each variable $x_i$
with either~$y_i$, or~$0$ (formally $0\cdot 1$), or~$-y_i$ (formally
$-1\cdot y_i$). We call~${\Phi^+}$ the new set of formulas, which, by
construction, is satisfiable in positive numbers $y_i = b_i$. To establish
the lemma, it suffices to find a solution of~${\Phi^+}$ taking values
in~$C^{\star}_{\Phi,d}$.

By Lemma~\ref{closure}, we have an assignment
$y_i = c_i$ of values $c_1\dotsc c_d$ in $C_{{\Phi^+},d} \subseteq C_{\Phi,d}$ that satisfies
simultaneously all formulas~$\bar\phi$ with~$\phi\in {\Phi^+}$.
Let $-d\le n_1\dotsc n_d\le d$ be integers such that for all
$i,j$
\begin{align*}
n_i < n_j \quad &\text{if and only if} \quad {\textstyle \frac{b_i}{c_i} <\frac{ b_j}{c_j}}\\
0 < n_i \quad &\text{if and only if} \quad {\textstyle 1 <\frac{b_i}{c_i}}\\
n_i < 0 \quad &\text{if and only if} \quad {\textstyle \frac{b_i}{c_i} < 1}
\end{align*}
Such numbers exist: simply sort the set
$\{1\}\cup\big\{\frac{b_i }{ c_i}\;|\;i=1\dotsc d\,\big\}$ and consider the positions in
the sorted sequence counting from that of~$1$.
We claim that the assignment
$y_i = c_i + n_ic_i\boldsymbol\epsilon\in\Q^\star$ satisfies all formulas
of~${\Phi^+}$.
To check this, we consider the different cases for atomic formulas
\begin{itemize}
\item $k\cdot y_i < h \cdot y_j$: if $k c_i < h c_j$ this is
obviously satisfied. Otherwise $k c_i = h c_j$, in this case $k$ and~$h$
are positive and the
constraint
\[kc_i+kn_ic_i\boldsymbol\epsilon < hc_j+hn_jc_j\boldsymbol\epsilon\]
is equivalent to~$n_i<n_j$. This, in turn, is equivalent by construction
to $\frac{b_i}{c_i} < \frac{b_j}{c_j}$ which we get by observing that $b_ihc_j = b_ikc_i <
b_jhc_i$.
\item $k\cdot y_i = h \cdot y_j$: obviously $k b_i = h b_j$ and $k c_i = h
c_j$, therefore $\frac{b_i}{c_i} = \frac{b_j}{c_j}$, and, as a consequence, also $n_i=n_j$
from which the statement.
\item $k\cdot 1 < h \cdot y_j$: similarly to the first case,
if $k<h c_j$ this is immediate. Otherwise $k = hc_j$, so $k$ and~$h$ are
positive, the constraint
\[k\cdot 1 < hc_j+hn_jc_j\boldsymbol\epsilon\]
is equivalent to
$0 < n_j$, in other words $1<\frac{b_j}{c_j}$, which follows observing that
$hc_j = k < hb_j$.
\item $k\cdot y_i < h \cdot 1$: as the case above.
\item $k\cdot 1 = h \cdot y_j$: obviously $k\cdot 1 = h b_j = h c_j$,
therefore $\frac{b_j}{c_j} = 1$, so $n_j=0$ and the case follows.
\item $k\cdot y_i = h \cdot 1$: as the case above.
\end{itemize}
\end{proof}

\begin{proof}[Proof of Theorem~\ref{thm:sampling}]
The sampling algorithm produces the finite
substructure~$\mfa^\star_{\atm{\tau},d}$ of~$\mfa^\star$ having
domain~$D_{\atm{\tau},d}$ where
$\atm{\tau}:=\bigcup_{R\in\tau}\atm{\phi_R}$, namely the $\tau$-structure with
domain~$D_{\atm{\tau},d}$ in which each relation symbol~$R\in\tau$
denotes the restriction of~$R^{\mfa^{\star}}$ to~$D_{\atm{\tau},d}$. It is
immediate to observe that this structure has size polynomial in~$d$.

Since $\mfa^\star_{\atm{\tau},d}$ is a substructure of~$\mfa^\star$, it is
clear that if an instance is satisfiable in~$\mfa^\star_{\atm{\tau},d}$,
then it is a fortiori satisfiable in~$\mfa^\star$.

The vice versa follows from Lemma~\ref{pippo}. In fact, consider a set
$\Psi$ of atomic $\tau$-formulas having free variables $x_1\dotsc x_d$. Assume
that $\Psi$ is satisfied in~$\mfa^\star$ by one assignment~$x_i=a_i$ for~$i\in \{1\dotsc d\}$. For
each~$R\in \Psi$ let $\Phi_R \subset \atm{\phi_R}$ be the set of atomic subformulas
of~$\phi_R$ which are satisfied by our assignment~$a_i$. Clearly the atomic
$\tau_0$-formulas $\Phi:=\bigcup_{R\in \Psi} \Phi_R$ are simultaneously satisfiable.
Remembering that the formulas~$\phi_R$ have no negations by construction,
it is obvious that any simultaneous solution of~$\Phi$ must also
satisfy~$\Psi$.
By Lemma~\ref{pippo}, $\Phi$ has a solution in the set~$D_{\Phi,d}$ defined
therein. We can observe
that $C_{\Phi,d}\subset C_{\atm{\tau},d}$,
hence $D_{\Phi,d}\subset D_{\atm{\tau},d}$ and the claim follows.
\end{proof}

%% file: VCSP.tex
\section{Tractability of Submodular PLH Valued Constraints}
\label{sect:tract}

Here we extend the method developed in Section~\ref{sect:csp-tract} to the
treatment of VCSPs. To better highlight the parallel with
Section~\ref{sect:csp-tract}, so that the reader already familiar with it
may quickly get an intuition of the arguments here, we will use identical
notations to represent corresponding objects. This choice has the drawback
that some symbols, notably~$\Q^\star$, need to be re-defined (the new
$\Q^\star$, for instance, will contain the old one). In this section, we
will sometimes skip details that can be borrowed unchanged from
Section~\ref{sect:csp-tract}.

Our goal is to prove the following result
\begin{theorem}\label{thm:tract}
Let $\g$ be a PLH valued finite constraint language.
Assume that all cost functions in~$\g$ are submodular.
Then $\vcsp(\g)$ is polynomial-time solvable.
\end{theorem}

Let us begin with the new definition of~$\Q^\star$.

\begin{definition}
We let $\Q^\star$ denote the ring~$\Q((\eps))$ of formal Laurent
power series in the indeterminate~$\eps$. Namely $\Q^\star$ is the set of
formal expressions
\[\sum_{i=-\infty}^{+\infty} a_i\eps^i\]
where $a_i\neq0$ for only finitely many {\it negative} values of~$i$.
Clearly $\Q$ is embedded in~$\Q^\star$.
The ring operations on $\Q^\star$ are defined as usual
\begin{align*}
\sum_{i=-\infty}^{+\infty} a_i\eps^i + \sum_{i=-\infty}^{+\infty} b_i\eps^i &=
\sum_{i=-\infty}^{+\infty} (a_i+b_i) \eps^i \\
\sum_{i=-\infty}^{+\infty} a_i\eps^i \cdot \sum_{i=-\infty}^{+\infty} b_i\eps^i &=
\sum_{i=-\infty}^{+\infty} \left(\sum_{j=-\infty}^{+\infty} a_jb_{i-j}\right) \eps^i
\end{align*}
where the sum in the product definition is always finite by the hypothesis
on $a_i,b_i$ with negative index~$i$. The order is the lexicographical
order induced by~$0<\eps\ll1$, namely
\[
\sum_{i=-\infty}^{+\infty} a_i\eps^i < \sum_{i=-\infty}^{+\infty} b_i\eps^i
\quad\text{iff}\quad \exists i\; a_i < b_i \wedge \forall j<i\; a_j=b_j
\text{.}
\]
It is well known that $\Q^\star$ is an ordered field, namely all non-zero
elements have a multiplicative inverse and the order is
compatible with the field operations. We define the following subsets
of~$\Q^\star$ for $m\le n$
\[\Q^\star_{m,n} := \left\{ \sum_{i=m}^n \boldsymbol\epsilon^i a_i
\;\middle|\; a_i\in\Q \right\} \subset \Q^\star
\]
\end{definition}

\begin{definition}
We define a new structure $\mathfrak L^\star$, which is both an extension
and an expansion of~$\mathfrak L$, having $\Q^\star$ as domain and
$\tau_1 := \tau_0\cup\{k\}_{k\in\Q^\star_{-1,1}}$
as signature,
where the interpretation of
symbols in~$\tau_0$ is formally the same as for~$\mathfrak L$ and
the symbols $k\in\Q^\star_{-1,1}$ denote constants (zero-ary functions).
\end{definition}
Notice that, for technical reasons, we allow only constants
in $\Q^\star_{-1,1}$. During the rest of this section,
$\tau_1$-formulas will be interpreted in the structure~$\mathfrak
L^\star$. We make on $\tau_1$-formulas the same assumptions of
Section~\ref{sect:csp-tract} (that atomic subformulas are non-trivial and
not negated), also $H(\Phi)$ and~$K(\Phi)$ where $\Phi$ is a set of atomic
$\tau_1$-formulas are defined similarly to Section~\ref{sect:csp-tract}.
Observe that the reduct of $\mathfrak L^\star$ obtained by
restricting the language to~$\tau_0$ is {\it elementarily equivalent}
to~$\mathfrak L$, namely it satisfies the same first-order sentences.

The following lemmas~\ref{closure2}, \ref{pippo2}, and~\ref{vcsp-sampling}
are analogues of Lemma~\ref{closure} and Lemma~\ref{pippo}.

\begin{lemma}\label{closure2}
Let $\Phi$ be a finite set of atomic $\tau_1$-formulas having free variables
in~$\{x_1\dotsc x_d\}$.
Call~$\bar\Phi$ the set~$\bigcup_{\phi\in\Phi}\bar\phi$.
Suppose that there is $0<r\in\Q^\star$ such that
all satisfying assignments of~$\bar\Phi$ in the domain~$\Q^\star$
also satisfy $0< x_i\le r$ for all~$i$. Let $u$, $\alpha_1\dotsc\alpha_d$
be elements of~$\Q^\star$. Assume that the formulas in~$\Phi$ are simultaneously
satisfiable by a point~$(x_1\dotsc x_d)\in \Q^\star$ such that
$\sum_i \alpha_ix_i < u$.
Let us define the set
\[
C_{\Phi,d} = 
\left\{|k|\prod_{i=1}^s |h_i| ^{e_i} \,\middle|\, k\in K(\Phi) ,\;
e_1\dotsc e_s \in {\mathbb Z},\; \sum_{r=1}^s \lvert e_r \rvert < d
\right\} \subseteq \Q^\star_{-1,1}
\]
where $h_1\dotsc h_s$ is an enumeration of the (finitely many) elements
of~$H(\Phi)$.
Then there is
a point in $(x_1'\dotsc x_d')\in C_{\Phi,d}^d\subseteq\Q^\star$
with $\sum_i \alpha_ix_i' < u$ that satisfies simultaneously
all~$\bar\phi$, for~$\phi\in\Phi$.
\end{lemma}
\begin{proof}
As in the proof of Lemma~\ref{closure} (to which we direct the reader for
many details) we take a maximal~$\gamma\le\beta$ such that there are
$\Psi_1,\Psi_2,\Psi_3$ with
\begin{align*}
\bar \Phi \mkern 7mu &= \{s_1 = s'_1, \dotsc, s_\alpha = s'_\alpha\} \cup
\{ t_1 \le t'_1, \dotsc, t_\beta \le t'_\beta \}\\
\Psi_1 &= \{s_1 = s'_1, \dotsc, s_\alpha = s'_\alpha\}\\
\Psi_2 &= \{ t_1 = t'_1, \dotsc, t_\gamma=t'_\gamma \}\\
\Psi_3 &= \{t_{\gamma+1} \le t'_{\gamma+1},\dotsc, t_\beta \le t'_\beta \}
\end{align*}
and $\Psi_1 \cup \Psi_2 \cup \Psi_3$ is satisfiable by an assignment
with~$\sum_i \alpha_ix_i < u$. As in the proof of Lemma~\ref{closure}
the set of solutions of~$\Psi_1\cup\Psi_2$ satisfying~$\sum_i \alpha_ix_i
< u$ is contained in the solutions of~$\Psi_3$. So, here too, it
suffices to show that there is a solution of~$\Psi_1\cup\Psi_2$
with~$\sum_i \alpha_ix_i < u$ taking values in~$C_{\Phi,d}$. The proof
of Lemma~\ref{closure} shows that there is a solution
of~$\Psi_1\cup\Psi_2$ taking values in~$C_{\Phi,d}$ without necessarily
meeting the requirement that~$\sum_i \alpha_ix_i < u$. We will prove that, in
fact, any such solution meets the additional constraint.

Let $x_i=a_i,b_i$ be two distinct satisfying assignments
for~$\Psi_1\cup\Psi_2$ such that $\sum_i \alpha_ia_i < u$ and
$\sum_i \alpha_ib_i \ge u$. We know that the first exists, and we assume
the second towards a contradiction. The two assignments must differ, so,
without loss of generality~$a_1\neq b_1$.
For $t\in\Q^\star$, with $t\ge0$, define the assignment
$x_i(t)=(1+t)a_i-tb_i$.
Since all constraints in~$\Psi_1\cup\Psi_2$ are
equalities, it is clear that the new assignment~$x_i(t)$
satisfies~$\Psi_1\cup\Psi_2$ for all~$t\in\Q^\star$.
Moreover, if $t\ge0$
\[
\sum_i \alpha_ix_i(t) \le \sum_i \alpha_ia_i - t\left(\sum_i \alpha_ib_i -
\sum_i \alpha_ia_i\right) < u
\]
Let $t=\frac{2r}{|b_1-a_1|}$. Then
\[
x_1(t) = a_1 + \frac{2r}{|b-a|}(a-b)
\]
is either $\ge 2r$ or $<0$ depending on the sign of~$(a-b)$. In either
case we have a solution~$x_i=x_i(t)$ of~$\Psi_1\cup\Psi_2$ satisfying~$\sum_i
\alpha_ix_i(t) < u$, which must therefore be a solution of~$\Phi$,
that does not satisfy $0< x_i\le r$.
\end{proof}

\begin{lemma}\label{pippo2}
Let $\Phi$ be a finite set of atomic $\tau_1$-formulas having free variables
in~$\{x_1\dotsc x_d\}$. Suppose that there are $0<l<r\in\Q^\star$ such that
all satisfying assignments of~$\Phi$ in the domain~$\Q^\star$
also satisfy $l< x_i< r$ for all~$i$. Let $\alpha_1\dotsc\alpha_d$
be rational numbers and~$u\in\Q^\star_{-1,1}$. Assume that the formulas in~$\Phi$ are simultaneously
satisfiable by a point~$(x_1\dotsc x_d)\in \Q^\star$ such that
$\sum_i \alpha_ix_i \le u$. Then the same formulas
are simultaneously satisfiable by a point $(x_1'\dotsc x_d')\in
(C^\star_{\Phi,d})^d\subseteq(\Q^\star)^d$ such that $\sum_i \alpha_ix_i' \le u$
where
\[
C^\star_{\Phi,d} = \{x + nx\boldsymbol\epsilon^3 \;|\; x \in C_{\Phi,d},\;
n\in\mathbb Z,\; -d\le
n\le d\} \subseteq \Q^\star_{-1,4} .
\]
\end{lemma}
\begin{proof}
We consider two cases: either
all satisfying assignments satisfy the inequality $\sum_i \alpha_ix_i \ge u$
or 
there is a satisfying assignment~$(x_1\dotsc
x_d)$ for~$\Phi$ such that $\sum_i \alpha_ix_i < u$.

In the first case, we
claim that all satisfying assignments, in fact, satisfy $\sum_i
\alpha_ix_i = u$. In fact, assume that $x_i=a_i,b_i$ are two satisfying
assignments such that $\sum_i \alpha_ia_i = u$ and~$v := \sum_i
\alpha_ib_i > u$. As in the proof of Lemma~\ref{closure2}, consider
assignments of the form~$x_i(t)=(1+t)a_i - t b_i$
for $t\in\Q^\star$. Clearly $\sum_i \alpha_ix_i(t) = u - t (v-u) < u$ for
all~$t>0$. As in Lemma~\ref{closure2}, the new assignment must
satisfy all equality constraints in~$\Phi$. Each inequality
constraint implies a strict inequality on~$t$ (remember that $\Phi$ only
has strict inequalities). Since all of these must be satisfied by~$t=0$,
there is an open interval of acceptable values of~$t$ around~$0$, and,
in particular, an acceptable~$t>0$. Our claim is thus established. 
Therefore, in this case,
it suffices to find any satisfying assignment for~$\Phi$
taking values in~$C^\star_{\Phi,d}$. The assignment is now constructed as in the
proof of Lemma~\ref{pippo}, replacing the formal symbol~$\boldsymbol\epsilon$ in that
proof by~$\boldsymbol\epsilon^3$. Namely take a satisfying
assignment~$x_i = b_i$ for~$\Phi$, and, by Lemma~\ref{closure2}, one
satisfying assignment~$x_i=c_i$ for~$\bar\Phi$ taking values
in~$C_{\Phi,d}$.
Observe that the hypothesis that all solutions of~$\Phi$ satisfy~$l<x_i$
for all~$i$ is used here to ensure that all solutions of~$\bar\Phi$ assign
positive values to the variables, which is required by
Lemma~\ref{closure2}.
Let $-d\le n_1\dotsc n_d\le d$ be integers such that for
all
$i,j$
\begin{align*}
n_i < n_j \quad &\text{if and only if} \quad {\textstyle \frac{b_i}{c_i}
<\frac{ b_j}{c_j}}\\
0 < n_i \quad &\text{if and only if} \quad {\textstyle 1
<\frac{b_i}{c_i}}\\
n_i < 0 \quad &\text{if and only if} \quad {\textstyle \frac{b_i}{c_i} <
1}
\end{align*}
The assignment
$y_i = c_i + n_ic_i\boldsymbol\epsilon^3$ can be seen to satisfy
all formulas of~$\Phi$ by the same check as in the proof of
Lemma~\ref{pippo}. Observe that we have to replace $\eps$
in Lemma~\ref{pippo} by~$\eps^3$ here, so
that~$\Q^\star_{-1,1} \cap \,\eps^3 \Q^\star_{-1,1} = \emptyset$.

For the second case, fix a satisfying assignment $x_i=b_i$.
By Lemma~\ref{closure2} there is an
assignment~$x_i=c_i\in C_{\Phi,d}$ such that $\sum_i \alpha_ic_i < u$
and this assignment satisfies~$\bar\phi$ for all~$\phi\in\Phi$. From these
two assignments construct the numbers~$n_i$ and then the assignment~$y_i =
c_i + n_ic_i\boldsymbol\epsilon^3$ as before. For the same reason
it is clear that the new assignment satisfies~$\Phi$. To conclude that
$\sum_i \alpha_iy_i < u$ we write
\[
\sum_i \alpha_iy_i = \sum_i \alpha_ic_i + \boldsymbol\epsilon^3
\sum_i \alpha_in_ic_i < u
\]
because the first summand is in~$\Q^\star_{-1,1}$ and~$<u$, therefore the
second summand is neglected in the lexicographical order.
\end{proof}

\begin{lemma}\label{vcsp-sampling}
Let $\Phi$ be a finite set of atomic $\tau_0$-formulas having free variables
in~$\{x_1\dotsc x_d\}$.
Let $u$, $\alpha_1\dotsc\alpha_d$
be rational numbers. Then the following are equivalent
\begin{enumerate}
\item The formulas in~$\Phi$ are simultaneously
satisfiable in $\Q$, by a point $(x_1\dotsc x_d)\in\Q^d$ such that
$\sum_i \alpha_ix_i \le u$.
\item The formulas in~$\Phi$ are simultaneously
satisfiable in $D_{\Phi,d}\subseteq\Q^\star$,
by a point~$(x_1'\dotsc x_d')\in D_{\Phi,d}^d$
such that $\sum_i \alpha_ix_i' \le u$,
where the set $D_{\Phi,d}$ is defined as follows
\begin{align*}
D_{\Phi,d} &:= -C^{\star}_{\Phi',d} \cup \{0\} \cup C^{\star}_{\Phi',d} 
\subseteq \Q^\star_{-1,4} \\
\Phi' &:= \Phi \cup \{x > \eps,\;x < -\eps,\;x > -\eps^{-1},\;x <
\eps^{-1}\}
\end{align*}
\end{enumerate}
\end{lemma}
\begin{proof}
The implication $\textit{2}\rightarrow\textit{1}$\/ is immediate observing
that the conditions~$\Phi$ and~$\sum_i \alpha_ix_i \le u$ are
first-order definable in~$\mathfrak S$. In fact, any assignments with
values in~$D_{\Phi,d}$ satisfying the conditions is, in particular, an
assignment in~$\Q^\star$, and, by completeness of the first order theory
of ordered $\Q$-vector spaces, we have an assignment taking values
in~$\Q$.

For the vice versa, fix any assignment $x_i=a_i$ with $a_i\in\Q$
for~$i\in\{1\dotsc d\}$. We pre-process the formulas in $\Phi$ producing a
new set of atomic formulas~$\Phi'$ as follows. We replace all variables
$x_i$ such that $a_i=0$ with the constant~$0=0\cdot1$. Then we replace
each of the remaining variables $x_i$ with either~$y_i$ or~$-y_i$
according to the sign of~$a_i$. Finally, we add the constraints $\eps<y_i$
and $y_i<\eps^{-1}$ for each of these variables. Similarly we produce
new coefficients $\alpha'_i = \text{sign}(a_i)\alpha_i$.
It is clear that
the new
set of formulas~$\Phi'$ has a satisfying assignment in positive rational
numbers with~$\sum_i \alpha'_iy_i \le u$.
Observing that
a positive rational~$x$ always satisfies~$\eps<x<\eps^{-1}$, we see that $\Phi'$ satisfies the
hypothesis of Lemma~\ref{pippo2} with $l=\eps$ and~$r=\eps^{-1}$. Hence the
statement.
\end{proof}

Two roads diverge now. Clearly the formulas~$\Phi$ in
Lemma~\ref{vcsp-sampling} are going to define a piece of the domain of a
piecewise linear homogeneous function, while the coefficients~$\alpha_i$
define the function on that piece. We could decide to interpret our PLH
functions in the domain~$\Q^\star$ or we could decide to substitute a
suitably small rational value of~$\eps$ in the formal expression
of~$D_{\Phi,d}$ and map the problem to~$\Q$. In the first case we have to
{\it transfer}\/ the known approaches for~$\Q$ to the new domain, in the second
case we can~{\it use}\/ them (after having computed a suitable~$\eps$). It
is not clear which road is the less
traveled by. For reasons that will be
discussed in Subsection~\ref{discussion} we take the one of transferring.

It is obvious that one can extend Definition~\ref{vcspdef} considering
$\vcsp$s whose cost functions take values in any totally ordered
ring containing~$\Q$, and in particular in~$\Q^\star$. We will need to
establish the basics of such extended VCSPs.
More precisely, we will need to prove Corollary~\ref{hereafter}
hereafter, that builds on a {\it fully combinatorial}
algorithm (Theorem~\ref{iworthm}) due to Iwata and
Orlin~\cite{IwataOrlin}.

\begin{definition}
Let $R$ be a totally ordered commutative ring with unit. 
A problem over $R$ can be solved in \emph{fully combinatorial
polynomial-time} if there exists a polynomial-time (uniform) machine on
$R$ in the sense of \cite{BCSS} (see Chapters 3-4; such a machine operates on strings of symbols that represent elements of an ordered commutative ring, rather than bits as in classical Turing machines) solving it by
performing only additions and comparisons of elements in $R$ as
fundamental operations. (Notice that in such a machine there are no
machine-constants except $1$.)
\end{definition}

\begin{definition}
A collection $\mathcal C$ of subsets of a given set $Q$ is said to
be a \emph{ring family} if it is closed under union and intersection.
\end{definition}

Equivalently, a ring family is a distributive sublattice of~$\mathcal P(Q)$ with respect to
union and intersection, notably every distributive lattice can be
represented in this form (Birkhoff's representation theorem).
Computationally, we represent a ring family
following~\cite[Section~6]{SCHRIJVER2000346}. Namely, fixed a
representation for the elements of~$Q$, the ring family~$\mathcal C$ is
represented by the smallest set~$M\subseteq Q$ in~$\mathcal C$, and an oracle that given an
element of~$v\in Q$ returns the smallest $M_v\subset Q$ in~$\mathcal C$
such that $v\in M_v$. The construction of Section~6
in~\cite{SCHRIJVER2000346} proves that any algorithm capable of
minimising submodular set functions can be used to minimise submodular set
functions defined on a ring family represented in this way. Observe that
this construction is fully combinatorial.

\begin{theorem}[Iwata-Orlin \cite{IwataOrlin} +
Schrijver~\cite{SCHRIJVER2000346}] \label{iworthm}
There exists a fully combinatorial polynomial-time algorithm over
$\Q$ that
\begin{itemize} 
\item taking as input a finite set $Q=\{1,\ldots,n\}$ and a ring family,
$\mathcal C \subseteq 2^Q$, represented as
in~\cite[Section~6]{SCHRIJVER2000346} (namely as above),
\item having access to an oracle computing a submodular set-function $\psi
\colon \mathcal C \rightarrow \Q$,
\end{itemize}
computes an element $S \in \mathcal C$ such that $\psi(S)=\min_{A\in \mathcal C}
\psi(A)$ in
time bounded by a polynomial $p(n)$ in the size $n$ of the domain.
\end{theorem}

\begin{corollary}\label{hereafter}
Let $R$ be a totally ordered commutative ring with unit (for
instance~$\Q^\star$),
there exists a fully combinatorial polynomial-time algorithm over~$R$
that
\begin{itemize} 
\item taking as input a finite set $Q=\{1,\ldots,n\}$ and a ring family,
$\mathcal C \subseteq 2^Q$, represented as in
Theorem~\ref{iworthm}, 
\item having access to an oracle computing a submodular set-function $\psi
\colon \mathcal C \rightarrow R$,
\end{itemize}
computes an element $S \in \mathcal C$ such that $\psi(S)=\min_{A\in \mathcal C}
\psi(A)$ in
time bounded by a polynomial $p(n)$ in the size $n$ of the domain.
\end{corollary}
\begin{proof}
Theorem~\ref{iworthm} provides a fully combinatorial algorithm to minimise
submodular functions that, over~$\Q$, runs in polynomial time and computes
a correct result. We claim that any such algorithm must be correct and run
in polynomial time over~$R$ as well.
To show this, we prove the following:
        \begin{enumerate}
                \item The algorithm terminates in time $p(n)$, where
$p(n)$ is as in Theorem \ref{iworthm}.
                \item The output of the algorithm coincides with the
minimum of $\psi$.
        \end{enumerate}
Let $R_\psi$ denote the subgroup of the additive group $(R,+)$ generated by
$\psi(\mathcal C)$, and let $E_\psi:=\{g_1,\ldots,g_m\}$  be a set of free
generators of $R_\psi$.
For any tuple $ r=(r_1,\ldots, r_m) \in
\Q^m$, we define a group homomorphism $h_{r}\colon
R_\psi \rightarrow \Q$, by $h_{r}(g_i)=r_i$. Let $R_N:=N\cdot(E_\psi \cup
\{0\}\cup -E_\psi)$ be the
subset of $R$ consisting of the elements of the form $\pm x_1 \pm  x_2
\ldots \pm x_k$, with
$k \leq N$, $x_1,x_2, \ldots, x_k \in E_\psi$.

In general, the group homomorphisms ${h_r}$ are not order preserving. We claim that for all $N$,
there exists $r \in \Q^m$ such that ${h_r}|_{R_N}$ is order preserving.
To see this, assume that no such tuple $r$
exists. The inequalities denoting that ${h_r}|_{R_N}$ is order
preserving are expressed by a finite linear program $P$ in the variables
$r_1, \ldots,r_m$. By the assumption and  Farkas' lemma
there is a linear
combination (with coefficients in $\mathbb Z$) of the inequalities of $P$
which is contradictory. Therefore $P$ is contradictory in any ordered ring, and, in particular, in $R$.
However $r_i=g_i$, for all $i \in \{1,\ldots, m\}$, is a valid
solution of~$P$ in $R$.

Fix $N:=\hat N\cdot 2^{p(n)}$, where $\hat{N}$ is such that
$\psi(S)\in R_{\hat N}$ for all $S \in \mathcal C$. For this $N$, let
$r$ be a tuple satisfying the claim. We run two parallel
instances of the algorithm, one over $R$ with input $\psi$, and the other
in $\Q$ with input ${h_ r}\circ \psi$. We can prove that the
two runs are \textit{exactly parallel} for at least $p(n)$ steps,
therefore, since the second run stops within these $p(n)$ steps, also the
first one must do so. Formally, we prove, in a register machine model,
that, at each step $i \leq p(n)$, if a register  contains the
value $g$ in the first run, it must contain the value ${h_r}(g)$ in the
second. This is easily established proving by induction on $i$ that a
value computed at step $i$ must be in $R_{\hat N\cdot2^i}$.
Point~1 is thus established.

For point~2, let $\min_R$ and $\min_{\Q}$ be the output of the
algorithm over $(\psi,R)$ and $({h_r} \circ \psi, \Q)$, respectively.
The induction above shows, in particular, that~$\min_{\Q}=h_r(\min_R)$.
We know that $h_r(\min_R)=\min_{\Q}={h_r} \circ \psi (S_0)$ for some $S_0$ and
${h_r} \circ \psi (S)\ge\min_{\Q}=h_r(\min_R)$ for each element~$S$ of~$\mathcal C$.
        By our choice of $N$, the corresponding relations, $\min_R=
\psi(S_0)$ and $\psi(S)\ge\min_R$ for each element~$S$ of~$\mathcal C$, must hold in $R$.
\end{proof}

The following lemma is essentially contained
in~\cite[Theorem~6.7]{cohen2006complexity},
except that we replace the set of values by an arbitrary totally ordered
commutative ring with unit~$R$.
To state the lemma properly, we need to observe that, given a submodular
function~$f$ defined on~$Q^d$, where~$Q=\{1,\dotsc,n\}$, we can associate
to it the following ring family~$\mathcal C_f\subseteq\mathcal
P(Q\times\{1,\dotsc,d\})$. For every $x=(x_1,\dotsc,x_d)\in Q^d$ define
\[
\mathcal C_x := \{(q,i)\;|\;q\in Q,\,q\le x_i\}\subseteq Q\times\{1,\dotsc,d\}
\]
then we let~$C_f$ be the union of~$C_x$ for all~$x$ such
that~$f(x)<+\infty$.

\begin{lemma}\label{lemmaringfamily}
Let $R\supseteq\Q$ be a totally ordered ring.
There exists a fully combinatorial polynomial-time algorithm over~$R$
that
\begin{itemize} 
\item taking as input a finite set $Q=\{1,\ldots,n\}$ and an integer~$d$,
\item having access to an oracle computing a partial submodular
$f \colon Q^d \rightarrow R$,
\item given the representation of~$\mathcal C_f$ as in
Theorem~\ref{iworthm},
\end{itemize}
computes an $x\in Q^d$ such that $f(x)$ is minimal, in time polynomial
in~$n$ and~$d$.
\end{lemma}
\begin{proof}
The problem reduces to minimising a submodular set-function on the ring
family~$\mathcal C_f$, for the details see the proof of Theorem~6.7
in~\cite{cohen2006complexity}.
\end{proof}
%

\begin{proof}[Proof of Theorem~\ref{thm:tract}.]
Similarly to the proof of Theorem~\ref{thm-csp}, we will use a sampling
technique. Namely, given an instance $I$ of~$\vcsp(\g)$,
we will employ Lemma~\ref{vcsp-sampling}
to fix a finite structure~$\g_I$, of size (and also representation size) polynomial in~$|V_I|$, having a
subset~$\Q^\star_I$ of~$\Q^\star_{-1,4}$ as domain, such
that the variables $V_I$ of~$I$ have an assignment in~$\Q$ having cost~$\le
u_I$ if and only if they have one in~$\Q^\star_I$.
Once we have~$\g_I$, we will conclude by
Lemma~\ref{lemmaringfamily}.

The structure~$\g_I$ obviously needs to have the same signature~$\tau$
as~$\g$. For each function symbol~$f\in\tau$ we consider a $\tau_0$-formula~$\phi_f$ defining~$f^{\g}$ and we let $f^{\g_I}$ be the
function defined in~$\Q^\star$ by the same formula. By
Proposition~\ref{equiv} the choice of~$\phi_f$ is immaterial. Remains to
define the domain~$\Q^\star_I\subset\Q^\star$.

By quantifier-elimination (Theorem~\ref{thm:qe}),
any piecewise linear homogeneous cost function $f \colon {\Q}^n \to \Q \cup \{
+\infty \}$ 
can be written as  
\[f(x_1,\dots,x_{\text{ar}(f)}) \;=\; \begin{cases}
t_{f,1} &  \text{ if } \chi_{f,1} \\
\cdots \\
t_{f,m_f} &  \text{ if } \chi_{f,m_f} \\
+\infty & \text{ otherwise}
\end{cases}\]
where $t_{f,1},\dots,t_{f,m_f}$ are $\tau_0$-terms, 
$\chi_{f,1},\dots,\chi_{f,m_f}$ are conjunctions of atomic $\tau_0$-formulas
with variables from $\{x_1,\dots,x_{\text{ar}(f)}\}$, and $\chi_{f,1},\dots,\chi_{f,m_f}$ define disjoint subsets of $\Q^n$. We fix such a
representation for each of the cost functions in $\g$, and we collect all
the atomic formulas appearing in every one of the
conjunctions~$\chi_{f,i}$, for $f\in\g$ and $1\le i\le m_f$, into the
set~$\Phi$. Clearly $\Phi$ is finite and depends only on the fixed
language~$\g$. Finally, $\Q^\star_I := D_{\Phi,|V_I|}$ as defined in
Lemma~\ref{vcsp-sampling}.

The size of~$\Q^\star_I$ is clearly polynomial by simple inspection of the
definition. Its representation has also polynomial size if the numbers are
represented in binary, and, with this representation, the evaluation of
$f^{\g_I}$ for $f\in\tau$ takes polynomial time.

Given an assignment $\alpha\colon V_I\to\Q^\star_I$ of value~$\le u_I$ we
have, a fortiori, an assignment~$\colon V_I\to\Q^\star$ of value~$\le
u_I$, hence, by the usual completeness of the first order theory of
ordered $\Q$-vector spaces, there is an assignment~$\colon V_I\to\Q$ with
the same property.

Finally let $\beta\colon V_I\to\Q$ be an assignment having value~$\le
u_I$. We need to find an assignment~$\beta'\colon V_I\to\Q^\star_I$ with
value~$\le u_I$. Let
\[
\phi_I = \sum_{i=1}^m f_i(x_1^i,\dotsc,x_{\text{ar}(f_i)}^i)
\]
(cf.\ Definition~\ref{vcspdef}). For each $i\in\{1,\dotsc,m\}$ select
the formula $\chi_i$ among $\chi_{f_i,1},\dotsc,\chi_{f_i,m_{f_i}}$ that
is satisfied by the assignment~$\beta$. Clearly, the conjunction of atomic
$\tau_0$-formulas~$\chi:=\bigwedge_{i=1}^m \chi_i$ is satisfiable. 
Moreover, $\phi_I$ restricted to the subset of~$(\Q^\star)^{|V_I|}$ where $\chi$
holds is obviously linear. Then we can apply Lemma~\ref{vcsp-sampling},
and we get an assignment~$\beta'$ whose values are in~$D_{\chi,|V_I|}$ (where,
by a slight abuse of notation, we wrote $\chi$ for the set of conjuncts of $\chi$). We conclude observing that $D_{\chi,|V_I|}
\subseteq D_{\Phi,|V_I|} = \Q^\star_I$.

It remains to check that Lemma~\ref{lemmaringfamily} applies to our
situation. Clearly $R=\Q^\star$, the function~$f$ is the objective function described by~$\phi_I$,
and we let $n=|\Q^\star_I|$ so that we identify $Q$ with an enumeration of~$\Q^\star_I$
in increasing order (which can be computed in polynomial time without
obstacle). The oracle computing~$f$ is straightforward to implement since
sums and comparisons in~$\Q^\star$ merely reduce to the corresponding
component-wise operations on the coefficients. The representation of the
ring family~$\mathcal C_f$ requires a moment of attention. To construct
the oracle, as well as to find the minimal element~$M$, we need an
algorithm that, given a variable~$x\in V_I$ and a value~$q\in\Q^\star_I$,
finds the component-wise minimal feasible assignment~$\alpha_x\colon
V_I\to \Q^\star_I$ that gives to~$x$ a value~$\ge q$ (which exists
observing that the set of feasible assignments is min-closed). This
algorithm is easy to construct observing that the feasibility problem is
a min-closed CSP. We describe how to find~$M$, the procedure for~$M_v$ is
essentially the same.

Suppose that for each variable $x\in V_I$ we can find the smallest
element~$\beta(x)\in\Q^\star_I$ such that there is a feasible
assignment~$\gamma_x\colon V_I\to \Q^\star_I$ such that
$\gamma_x(x)=\beta(x)$, then, by the min-closure, $\beta = \min_{x\in
V_I} \gamma_x$ is the minimal assignment. To find $\beta(x)$ it is
sufficient to solve the feasibility problem, using Theorem~\ref{thm-csp},
adding a constraint~$x\ge k$ for increasing values of~$k\in \Q^\star_I$.
\end{proof}

\subsection{Why \mathshift\Q^\star\mathshift?}\label{discussion}

It might appear that in more than one occasion we chose to work in
mathematically overcomplicated structures. For example, the algorithm for
Theorem~\ref{thm:tract} merely manipulates points in~$\Q^\star_{-1,4}$,
which is just~$\Q^6$ with the lexicographic order, yet we went to the
trouble of introducing the field of formal Laurent power series. More
radically, one might observe that assigning a rational value to the formal
variable~$\eps$ small enough, we could have mapped the entire algorithm
to~$\Q$, thus dispensing with non-Archimedean extensions entirely. As we
believe to owe to our reader an explanation for this, we better give three.

First, the idea of limiting our horizon to~$\Q^\star_{-1,4}\simeq\Q^6$
might seem a simplification, but, in practice, it makes things more
complicated. For example, in several places we used the fact $\Q^\star$
has a field structure to make proofs more direct and intuitive. Second,
going for the most elementary exposition, namely choosing an~$\eps$ small
enough explicitly, would have completely obfuscated any idea in the
arguments, which would have been converted in some unsightly bureaucracy
of inequalities. Even computationally, mapping everything to~$\Q$ is
tantamount as converting arrays of small integers into bignums by
concatenation, hardly an improvement. Finally, the existence of an
efficiently computable rational value of~$\eps$ that works is not
necessary for our method, even though, in this case, a posteriori, such
an~$\eps$ exists.

Our third, and most important, justification, is that we desire to present
the approach used in this paper, which is quite generic, as much as the
results. To this aim, it is convenient to express the underlying ideas in
their natural language. For example, Corollary~\ref{hereafter} is a
completely black-boxed way to transfer combinatorial algorithms between
domains that share some algebraic structure. We do not claim great
originality in that observation, yet we believe that the method is
interesting, and worthy of being presented in the cleanest form that we
could devise.